\documentclass[12pt,reqno]{article}

\usepackage[usenames]{color}
\usepackage{amssymb}
\usepackage{graphicx}
\usepackage{amscd}

\usepackage{amsthm}
\newtheorem{theorem}{Theorem}

\newtheorem{proposition}[theorem]{Proposition}
\newtheorem{corollary}[theorem]{Corollary}

\theoremstyle{definition}

\newtheorem{example}[theorem]{Example}

\usepackage[colorlinks=true,
linkcolor=webgreen, filecolor=webbrown,
citecolor=webgreen]{hyperref}

\definecolor{webgreen}{rgb}{0,.5,0}
\definecolor{webbrown}{rgb}{.6,0,0}

\usepackage{color}

\usepackage{float}

\usepackage{graphics,amsmath,amssymb}
\usepackage{amsfonts}
\usepackage{latexsym}
\usepackage{epsf}

\newcommand{\seqnum}[1]{\href{http://www.research.att.com/cgi-bin/access.cgi/as/~njas/sequences/eisA.cgi?Anum=#1}{\underline{#1}}}

\begin{document}

\begin{center}
\vskip 1cm{\LARGE\bf  Combinatorial polynomials as moments, Hankel transforms and exponential Riordan arrays}  \vskip 1cm
\large
Paul Barry\\
School of Science\\
Waterford Institute of Technology\\
Ireland\\
\href{mailto:pbarry@wit.ie}{\tt pbarry@wit.ie} \\
\end{center}
\vskip .2 in

\begin{abstract} In the case of two combinatorial polynomials, we show that they can exhibited as moments of paramaterized families of
orthogonal polynomials, and hence derive their Hankel transforms. Exponential Riordan arrays are the main vehicles used for this.
\end{abstract}
\section{Introduction}
Let $[n]={1,2,\ldots,n}$, and let $\mathsf{SP}_n$ be the set of set-partitions of $[n]$. For a set-partition
$\pi \in \mathsf{SP}_n$, let $|\pi|$ be the number of parts in $\pi$. Then the $n$-th exponential polynomial, also known as the $n$-th
Touchard polynomial (and sometimes called the $n$-th Bell polynomial \cite{Bell}), is given by
$$e_n(z)=\sum_{\pi} z^{|\pi|}=\sum_{k=0}^n S(n,k)z^k,$$ where $$S(n,k)=\frac{1}{k!}\sum_{j=0}^k (-1)^{k-j}\binom{k}{j}j^n$$ is the
general element of the exponential Riordan array
$$\left[1, e^x-1\right].$$ This is the matrix of Stirling numbers of the second kind \seqnum{A008277}, which begins
\begin{displaymath}\left(\begin{array}{ccccccc} 1 & 0 &
0
& 0 & 0 & 0 & \ldots \\0 & 1 & 0 & 0 & 0 & 0 & \ldots \\ 0 & 1
& 1 & 0 & 0 &
0 & \ldots \\ 0 & 1 & 3 & 1 & 0 & 0 & \ldots \\ 0 & 1 & 7
& 6 & 1 & 0 & \ldots \\0 & 1 & 15 & 25 & 10 & 1
&\ldots\\
\vdots &
\vdots & \vdots & \vdots & \vdots & \vdots &
\ddots\end{array}\right).\end{displaymath}
It is well known \cite{Kratt, Rad, Siva} that the Hankel transform of these polynomials is given by
$$z^{\binom{n+1}{2}}\prod_{k=1}^n k!.$$

\noindent Now let $$A(n,k)=\sum_{k=0}^k (-1)^j (k-j)^n \binom{n+1}{j}$$ be the general term of the triangle of Eulerian numbers. The
matrix of these numbers \seqnum{A008292} begins
\begin{displaymath}\left(\begin{array}{ccccccc} 1 & 0 &
0
& 0 & 0 & 0 & \ldots \\0 & 1 & 0 & 0 & 0 & 0 & \ldots \\ 0 & 1
& 1 & 0 & 0 &
0 & \ldots \\ 0 & 1 & 4 & 1 & 0 & 0 & \ldots \\ 0 & 1 & 11
& 11 & 1 & 0 & \ldots \\0 & 1 & 26 & 66 & 26 & 1
&\ldots\\
\vdots &
\vdots & \vdots & \vdots & \vdots & \vdots &
\ddots\end{array}\right).\end{displaymath} $A(n,k)$ is the number of permutations in $\mathfrak{S}_n$ with $k$ excedances. The Eulerian
polynomials $\mathsf{EU}_n(z)$ are defined by
$$\mathsf{EU}_n(z)=\sum_{k=0}^n A(n,k)z^k.$$
It is shown in \cite{Siva} that the Hankel transform of these polynomials is given by
$$z^{\binom{n+1}{2}}\prod_{k=1}^n k!^2.$$ These two results are consequences of the following two theorems.
\begin{theorem}\label{Thm1}
The polynomials $e_n(z)$ are moments of the family of orthogonal polynomials whose coefficient array is given by the
inverse of the exponential Riordan array
$$\left[e^{z(e^x-1)}, e^x-1\right].$$
\end{theorem}
\begin{theorem} \label{Thm2} The polynomials $\mathsf{EU}_n(z)$ are moments of the family of orthogonal polynomials whose coefficient
array is given by the
inverse of the exponential Riordan array
$$\left[\frac{e^{zx}(1-z)}{e^{zx}-ze^x},\frac{e^{x}-e^{zx}}{e^{zx}-ze^x}\right].$$
\end{theorem}
Note that in the case $z=1$, the above matrix is taken to be $\left[\frac{1}{1-x},\frac{x}{1-x}\right]$, whose inverse is the
coefficient array of the Laguerre polynomials \cite{Lah}.

While partly expository in nature, this note assumes a certain
familiarity with integer sequences, generating functions, orthogonal polynomials \cite{Chihara, wgautschi, Szego}, Riordan arrays
\cite{SGWW, Spru}, production matrices \cite{ProdMat, P_W}, and the integer Hankel transform \cite{BRP, CRI, Layman}. Many interesting
examples of sequences and Riordan arrays can be found in Neil Sloane's On-Line Encyclopedia of Integer Sequences (OEIS), \cite{SL1,
SL2}. Sequences are frequently referred to by their OEIS number. For instance, the binomial matrix $\mathbf{B}$ (``Pascal's
triangle'')
is \seqnum{A007318}.

\noindent The plan of the paper is as follows: \begin{enumerate} \item This Introduction \item
Integer
sequences, Hankel transforms, exponential Riordan arrays, orthogonal polynomials \item Proof of Theorem \ref{Thm1}
\item Proof of Theorem \ref{Thm2} \end{enumerate}
\section{Integer
sequences, Hankel transforms, exponential Riordan arrays, orthogonal polynomials} In this section, we recall known results on integer
sequences, Hankel transforms, exponential Riordan arrays and orthogonal polynomials that will be useful for the sequel.

For an integer sequence $a_n$, that is, an element of
$\mathbb{Z}^\mathbb{N}$, the power series $f_o(x)=\sum_{k=0}^{\infty}a_k x^k$ is called the \emph{ordinary generating function} or
g.f.
of the sequence, while $f_e(x)=\sum_{k=0}^{\infty}\frac{a_k}{k!} x^k$ is called the \emph{exponential generating function} or e.g.f.
of
the sequence. $a_n$ is thus the coefficient of $x^n$ in $f_o(x)$. We denote this by $a_n=[x^n]f_o(x)$. Similarly, $a_n=n![x^n]f_e(x)$.
For instance, $F_n=[x^n]\frac{x}{1-x-x^2}$ is the $n$-th Fibonacci number \seqnum{A000045}, while $n!=n![x^n]\frac{1}{1-x}$, which
says
that $\frac{1}{1-x}$ is the e.g.f. of $n!$ \seqnum{A000142}. For a power series $f(x)=\sum_{n=0}^{\infty}a_n x^n$ with $f(0)=0$ and
$f'(0)\neq 0$ we
define the reversion or compositional inverse of $f$ to be the power series $\bar{f}(x)=f^{[-1]}(x)$ such that $f(\bar{f}(x))=x$. We
sometimes write $\bar{f}= \text{Rev}f$.

The {\it
Hankel
transform} \cite{Layman} of a given sequence
$A=\{a_0,a_1,a_2,...\}$ is the
sequence of Hankel determinants $\{h_0, h_1, h_2,\dots \}$
where
$h_{n}=|a_{i+j}|_{i,j=0}^{n}$, i.e

\begin{center} \begin{equation}
 \label{gen1}
 A=\{a_n\}_{n\in\mathbb N_0}\quad \rightarrow \quad
 h=\{h_n\}_{n\in\mathbb N_0}:\quad
h_n=\left| \begin{array}{ccccc}
 a_0\ & a_1\  & \cdots & a_n  &  \\
 a_1\ & a_2\  &        & a_{n+1}  \\
\vdots &      & \ddots &          \\
 a_n\ & a_{n+1}\ &    & a_{2n}
\end{array} \right|. \end{equation} \end{center} The Hankel
transform of a sequence $a_n$ and its binomial transform are
equal.

In the case that $a_n$ has g.f. $g(x)$ expressible in the form
$$g(x)=\cfrac{a_0}{1-\alpha_0 x-
\cfrac{\beta_1 x^2}{1-\alpha_1 x-
\cfrac{\beta_2 x^2}{1-\alpha_2 x-
\cfrac{\beta_3 x^2}{1-\alpha_3 x-\cdots}}}}$$ (with $\beta_i \neq 0$ for all $i$) then
we have \cite{Kratt, Kratt1, Wall}
\begin{equation}h_n = a_0^n \beta_1^{n-1}\beta_2^{n-2}\cdots \beta_{n-1}^2\beta_n=a_0^n\prod_{k=1}^n
\beta_k^{n-k+1}.\end{equation}
Note that this is independent from $\alpha_n$. In general $\alpha_n$ and $\beta_n$ are not integers.
Such a continued fraction is associated to a monic family of orthogonal polynomials which obey the three term recurrence
$$p_{n+1}(x)=(x-\alpha_n)p_n(x)-\beta_n p_{n-1}(x), \qquad
p_0(x)=1,\qquad p_1(x)=x-\alpha_0.$$
\noindent The terms appearing in the first column of the inverse of the coefficient array of these polynomials are the moments of
family.

The \emph{exponential Riordan group} \cite {PasTri,DeutschShap,ProdMat}, is a set of infinite
lower-triangular integer matrices, where each matrix is defined by a pair of generating functions $g(x)=g_0+g_1x+g_2x^2+\ldots$ and
$f(x)=f_1x+f_2x^2+\ldots$ where $f_1\ne 0$. The associated matrix is the matrix whose $i$-th column has exponential generating
function
$g(x)f(x)^i/i!$ (the first column being indexed by 0). The matrix corresponding to the pair $f, g$ is denoted by $[g, f]$. It is
\emph{monic} if $g_0=1$. The group law is  given by \begin{displaymath} [g, f]*[h, l]=[g(h\circ f), l\circ f].\end{displaymath} The
identity for this law is $I=[1,x]$ and the inverse of $[g, f]$ is $[g, f]^{-1}=[1/(g\circ \bar{f}), \bar{f}]$ where $\bar{f}$ is the
compositional inverse of $f$. We use the notation $\mathit{e}\mathcal{R}$ to denote this group. If $\mathbf{M}$ is the matrix $[g,f]$,
and $\mathbf{u}=(u_n)_{n \ge 0}$ is an integer sequence with exponential generating function $\mathcal{U}$ $(x)$, then the sequence
$\mathbf{M}\mathbf{u}$ has exponential generating function $g(x)\mathcal{U}(f(x))$. Thus the row sums of the array $[g,f]$ are given
by
$g(x)e^{f(x)}$ since the sequence $1,1,1,\ldots$ has exponential generating function $e^x$. \begin{example} The \emph{binomial matrix}
is the matrix with general term $\binom{n}{k}$. It is realized by Pascal's triangle. As an exponential Riordan array, it is given by
$[e^x,x]$. We further have $$([e^x,x])^m=[e^{mx},x].$$
\end{example}
\begin{example} We have
$$\left[ e^{z(e^x-1)}, e^x-1\right]=\left[e^{z(e^x-1)},x\right]\cdot \left[1, e^x-1\right].$$
\end{example} A more interesting factorization is given by
\begin{proposition} The general term of the matrix $\mathbf{L}=\left[ e^{z(e^x-1)}, e^x-1\right]$ is given by
$$L_{n,k}=\sum_{j=0}^n S(n,j)\binom{j}{k}z^{j-k}.$$
\end{proposition}
\begin{proof} A straight-forward calculation shows that
$$\left[ e^{z(e^x-1)}, e^x-1\right]=\left[1,e^x-1\right]\cdot \left[e^{zx},x\right].$$
The assertion now follows since the general term of $\left[1,e^x-1\right]$ is $S(n,k)$ and that of
$\left[e^{zx},x\right]$ is $\binom{n}{k}z^{n-k}$.
\end{proof}

\noindent As an example of the calculation of an inverse, we have
the following proposition.

\begin{proposition} $$\left[e^{z(e^x-1)},e^x-1\right]^{-1}=\left[e^{-zx},\ln(1+x)\right].$$

 \end{proposition} \begin{proof} This follows since with $$f(x)=e^x-1$$ we have $$\bar{f}(x)=\ln(1+x).$$ \end{proof}

\begin{proposition}
$$\left[\frac{e^{zx}(1-z)}{e^{zx}-ze^x},\frac{e^{x}-e^{zx}}{e^{zx}-ze^x}\right]^{-1}=\left[1+zx,
\frac{1}{z-1}\ln\left(\frac{1+zx}{1+x}\right)\right].$$ Note that in the case $z=1$, we have
$$\left[\frac{1}{1-x},\frac{x}{1-x}\right]^{-1}=\left[\frac{1}{1+x},\frac{x}{1+x}\right].$$
\end{proposition}
\begin{proof} This follows since with $$f(x)=\frac{e^{zx}(1-z)}{e^{zx}-ze^x}$$ we have
$$\bar{f}(x)=\frac{1}{z-1}\ln\left(\frac{1+zx}{1+x}\right).$$ \end{proof}

An important concept for the sequel is that of production matrix. The concept of a \emph{production matrix}
\cite{ProdMat_0, ProdMat} is a general one, but for this note we find it convenient to review it in the context of Riordan arrays.
Thus
let $P$ be an infinite matrix (most often it will have integer entries). Letting $\mathbf{r}_0$ be the row vector
$$\mathbf{r}_0=(1,0,0,0,\ldots),$$ we define $\mathbf{r}_i=\mathbf{r}_{i-1}P$, $i \ge 1$. Stacking these rows leads to another
infinite
matrix which we denote by $A_P$. Then $P$ is said to be the \emph{production matrix} for $A_P$. \noindent If we let
$$u^T=(1,0,0,0,\ldots,0,\ldots)$$ then we have $$A_P=\left(\begin{array}{c} u^T\\u^TP\\u^TP^2\\\vdots\end{array}\right)$$ and
$$DA_P=A_PP$$ where $D=(\delta_{i,j+1})_{i,j \ge 0}$ (where $\delta$ is the usual Kronecker symbol). In \cite{P_W} $P$ is
called the Stieltjes matrix associated to $A_P$. In \cite{ProdMat}, we find the following result concerning matrices that are
production matrices for exponential Riordan arrays. \begin{proposition} Let $A=\left(a_{n,k}\right)_{n,k \ge 0}=[g(x),f(x)]$ be an
exponential Riordan array and let \begin{equation}\label{seq_def} c(y)=c_0 + c_1 y +c_2 y^2 + \ldots, \qquad r(y)=r_0 + r_1 y + r_2
y^2
+ \ldots\end{equation} be two formal power series that that \begin{eqnarray}\label{r_def} r(f(x))&=&f'(x) \\ \label{c_def}
c(f(x))&=&\frac{g'(x)}{g(x)}. \end{eqnarray} Then \begin{eqnarray} (i)\qquad a_{n+1,0}&=&\sum_{i} i! c_i a_{n,i} \\ (ii)\qquad
a_{n+1,k}&=& r_0 a_{n,k-1}+\frac{1}{k!} \sum_{i\ge k}i!(c_{i-k}+k r_{i-k+1})a_{n,i} \end{eqnarray}  or, defining $c_{-1}=0$,
\begin{equation}\label{array_def} a_{n+1,k}=\frac{1}{k!}\sum_{i\ge k-1} i!(c_{i-k}+k r_{i-k+1})a_{n,i}.\end{equation} Conversely,
starting from the sequences defined by (\ref{seq_def}), the infinite array $\left(a_{n,k}\right)_{n,k\ge 0}$ defined by
(\ref{array_def}) is an exponential Riordan array. \end{proposition} \noindent A consequence of this proposition is that
$P=\left(p_{i,j}\right)_{i,j\ge 0}$ where $$p_{i,j}=\frac{i!}{j!}(c_{i-j}+jr_{r-j+1})  \qquad (c_{-1}=0).$$ Furthermore, the bivariate
exponential generating function $$\phi_P(t,z)=\sum_{n,k} p_{n,k}t^k \frac{z^n}{n!}$$ of the matrix $P$ is given by $$\phi_P(t,z) =
e^{tz}(c(z)+t r(z)).$$ Note in particular that we have $$r(x)=f'(\bar{f}(x))$$ and $$c(x)=\frac{g'(\bar{f}(x))}{g(\bar{f}(x))}.$$
\begin{example} We consider the exponential Riordan array $[\frac{1}{1-x},x]$, \seqnum{A094587}. This array \cite{Lah} has elements
\begin{displaymath}\left(\begin{array}{ccccccc} 1 & 0 & 0 & 0 & 0 & 0 & \ldots \\1 & 1 & 0 & 0 & 0 & 0 & \ldots \\ 2 & 2 & 1 & 0 & 0 &
0 & \ldots \\ 6 & 6 & 3 & 1 & 0 & 0 & \ldots \\ 24 & 24 & 12 & 4 & 1 & 0 & \ldots \\120 & 120  & 60 & 20 & 5 & 1 &\ldots\\ \vdots &
\vdots & \vdots & \vdots & \vdots & \vdots & \ddots\end{array}\right)\end{displaymath} and general term $[k \le n] \frac{n!}{k!}$ with
inverse \begin{displaymath}\left(\begin{array}{ccccccc} 1 & 0 & 0 & 0 & 0 & 0 & \ldots \\-1 & 1 & 0 & 0 & 0 & 0 & \ldots \\ 0 & -2 & 1
& 0 & 0 & 0 & \ldots \\ 0 & 0 & -3 & 1 & 0 & 0 & \ldots \\ 0 & 0 & 0& -4 & 1 & 0 & \ldots \\0 & 0  & 0 & 0 & -5 & 1 &\ldots\\ \vdots &
\vdots & \vdots & \vdots & \vdots & \vdots & \ddots\end{array}\right)\end{displaymath} which is the array $[1-x,x]$. In particular, we
note that the row sums of the inverse, which begin $1,0,-1,-2,-3,\ldots$ (that is, $1-n$), have e.g.f. $(1-x)\exp(x)$. This sequence
is
thus the binomial transform of the sequence with e.g.f. $(1-x)$ (which is the sequence starting $1,-1,0,0,0,\ldots$). In order to
calculate the production matrix $\mathbf{P}$ of $[\frac{1}{1-x},x]$ we note that $f(x)=x$, and hence we have $f'(x)=1$ so
$f'(\bar{f}(x))=1$. Also $g(x)=\frac{1}{1-x}$ leads to $g'(x)=\frac{1}{(1-x)^2}$, and so, since $\bar{f}({x})=x$, we get
$$\frac{g'(\bar{f}(x))}{g(\bar{f}(x))}=\frac{1}{1-x}.$$ Thus the generating function for $\mathbf{P}$ is
$$e^{tz}\left(\frac{1}{1-z}+t\right).$$ Thus $\mathbf{P}$ is the matrix $[\frac{1}{1-x},x]$ with its first row removed. \end{example}
\begin{example} We consider the exponential Riordan array $[1, \frac{x}{1-x}]$. The general term of this matrix \cite{Lah} may be
calculated as
follows: \begin{eqnarray*}T_{n,k}&=&\frac{n!}{k!}[x^n]\frac{x^k}{(1-x)^k}\\ &=&\frac{n!}{k!}[x^{n-k}](1-x)^{-k}\\
&=&\frac{n!}{k!}[x^{n-k}]\sum_{j=0}^{\infty}\binom{-k}{j}(-1)^jx^j\\ &=&\frac{n!}{k!}[x^{n-k}]\sum_{j=0}^{\infty}\binom{k+j-1}{j}x^j\\
&=&\frac{n!}{k!}\binom{k+n-k-1}{n-k}\\ &=&\frac{n!}{k!}\binom{n-1}{n-k}.\end{eqnarray*} Thus its row sums, which have e.g.f. $\exp
\left(\frac{x}{1-x}\right)$, have general term $\sum_{k=0}^n \frac{n!}{k!}\binom{n-1}{n-k}$. This is \seqnum{A000262}, the `number of
``sets of lists": the number of partitions of $\{1,..,n\}$ into any number of lists, where a list means an ordered subset'. Its
general
term is equal to $(n-1)!L_{n-1}(1,-1)$. The inverse of $\left[1, \frac{x}{1-x}\right]$ is the exponential Riordan array
$\left[1,\frac{x}{1+x}\right]$, \seqnum{A111596}. The row sums of this sequence have e.g.f. $\exp\left(\frac{x}{1+x}\right)$, and
start
$1, 1, -1, 1, 1, -19, 151, \ldots$. This is \seqnum{A111884}. To calculate the production matrix of $\left[1,\frac{x}{1+x}\right]$ we
note that $g'(x)=0$, while $\bar{f}(x)=\frac{x}{1+x}$ with $f'(x)=\frac{1}{(1+x)^2}$. Thus $$f'(\bar{f}(x))=(1+x)^2,$$ and so the
generating function of the production matrix is given by $$ e^{tz}t(1+z)^2.$$ \noindent The production matrix of the inverse
begins \begin{displaymath}\left(\begin{array}{ccccccc} 0 & 1 & 0 & 0 & 0 & 0 & \ldots \\0 & 2 & 1 & 0 & 0 & 0 & \ldots \\ 0 & 2 & 4 &
1
& 0 & 0 & \ldots \\ 0 & 0 & 6 & 6 & 1 & 0 & \ldots \\0 & 0 & 0 & 12 & 8 & 1 & \ldots \\0 & 0 & 0 & 0 & 20 & 10 &\ldots\\ \vdots &
\vdots & \vdots & \vdots & \vdots & \vdots & \ddots\end{array}\right).\end{displaymath} \end{example} \begin{example} The exponential
Riordan array $\mathbf{A}= \left[\frac{1}{1-x},\frac{x}{1-x}\right]$, or \begin{displaymath}\left(\begin{array}{ccccccc} 1 & 0 & 0 & 0
& 0 & 0 & \ldots \\1 & 1 & 0 & 0 & 0 & 0 & \ldots \\ 2 & 4 & 1 & 0 & 0 & 0 & \ldots \\ 6 & 18 & 9 & 1 & 0 & 0 & \ldots \\ 24 & 96 & 72
& 16 & 1 & 0 & \ldots \\120 & 600  & 600 & 200 & 25 & 1 &\ldots\\ \vdots & \vdots & \vdots & \vdots & \vdots & \vdots &
\ddots\end{array}\right),\end{displaymath} has general term $$ T_{n,k}=\frac{n!}{k!}\binom{n}{k}.$$  Its inverse is
$\left[\frac{1}{1+x},\frac{x}{1+x}\right]$ with general term
$(-1)^{n-k}\frac{n!}{k!}\binom{n}{k}$. This is \seqnum{A021009}, the triangle of coefficients of the Laguerre polynomials $L_n(x)$.
The
production matrix $\left[\frac{1}{1-x},\frac{x}{1-x}\right]$  is given by \begin{displaymath}\left(\begin{array}{ccccccc} 1 & 1 & 0 & 0
& 0 & 0 &
\ldots \\1 & 3 & 1 & 0 & 0 & 0 & \ldots \\ 0 & 4 & 5 & 1 & 0 & 0 & \ldots \\ 0 & 0 & 9 & 7 & 1 & 0 & \ldots \\0 & 0 & 0 & 16 & 9 & 1 &
\ldots \\0 & 0 & 0 & 0 & 25 & 11 &\ldots\\ \vdots & \vdots & \vdots & \vdots & \vdots & \vdots &
\ddots\end{array}\right).\end{displaymath} \end{example} \begin{example} The exponential Riordan array
$\left[e^x,\ln\left(\frac{1}{1-x}\right)\right]$, or \begin{displaymath}\left(\begin{array}{ccccccc} 1 & 0 & 0 & 0 & 0 & 0 & \ldots
\\1
& 1 & 0 & 0 & 0 & 0 & \ldots \\ 1 & 3 & 1 & 0 & 0 & 0 & \ldots \\ 1 & 8 & 6 & 1 & 0 & 0 & \ldots \\ 1 & 24 & 29 & 10 & 1 & 0 & \ldots
\\1 & 89 & 145 & 75 & 15 & 1 &\ldots\\ \vdots & \vdots & \vdots & \vdots & \vdots & \vdots & \ddots\end{array}\right)\end{displaymath}
is the coefficient array for the polynomials $$_2F_0(-n,x;-1)$$ which are an unsigned version of the Charlier polynomials (of order
$0$) \cite{wgautschi, Roman, Szego}. This is \seqnum{A094816}. It is equal to $$[e^x,x]\left[ 1,
\ln\left(\frac{1}{1-x}\right)\right],$$ or the product of the binomial array $\mathbf{B}$ and the array of (unsigned) Stirling numbers
of the first kind. The production matrix of the inverse of this matrix is given by \begin{displaymath}\left(\begin{array}{ccccccc} -1
&
1 & 0 & 0 & 0 & 0 & \ldots \\1 & -2 & 1 & 0 & 0 & 0 & \ldots \\ 0 & 2 & -3 & 1 & 0 & 0 & \ldots \\ 0 & 0 & 3 & -4 & 1 & 0 & \ldots \\0
& 0 & 0 & 4 & -5 & 1 & \ldots \\0 & 0 & 0 & 0 & 5 & -6 &\ldots\\ \vdots & \vdots & \vdots & \vdots & \vdots & \vdots &
\ddots\end{array}\right)\end{displaymath} which indicates the orthogonal nature of these polynomials. We can prove this as follows. We
have $$\left[e^x, \ln\left(\frac{1}{1-x}\right)\right]^{-1}=\left[e^{-(1-e^{-x})},1-e^{-x}\right].$$ Hence $g(x)=e^{-(1-e^{-x})}$ and
$f(x)=1-e^{-x}$. We are thus led to the equations \begin{eqnarray*} r(1-e^{-x})&=&\,e^{-x},\\ c(1-e^{-x})&=&-e^{-x},\end{eqnarray*}
with solutions $r(x)=1-x$, $c(x)=x-1$. Thus the bivariate generating function for the production matrix of the inverse array is
$$e^{tz}(z-1+t(1-z)),$$ which is what is required. \end{example}

\section{Proof of Theorem \ref{Thm1}}

\begin{proof} We show first that with $\mathbf{L}=\left[e^{z(e^x-1)},e^x-1\right]$, the matrix $\mathbf{L}^{-1}$ which is given by
$$\mathbf{L}^{-1}=\left[e^{z(e^x-1)},e^x-1\right]^{-1}=\left[e^{-zx},\ln(1+x)\right],$$ is the coefficient array of a family of
orthogonal polynomials. To this end, we calculate the production array of $\left[e^{z(e^x-1)},e^x-1\right]$.
We have $f(x)=e^x-1$, $f'(x)=e^x$ and $\bar{f}(x)=\ln(1+x)$. Thus
$$c(x)=f'(\bar{f}(x))=1+x.$$
Similarly, for $g(x)=e^{z(e^x-1)}$, we have $g'(x)=ze^{z(e^x-1)+x}$ and so
$$r(x)=\frac{g'(\bar{f}(x))}{g(\bar{f}(x))}=\frac{ze^{zx}(1+x)}{e^{zx}}=z(1+x).$$
Thus the production matrix sought has generating function
$$e^{tw}(c(w)+t r(w))=e^{tw}(1+w+t(z(1+w))).$$ Thus the production array $\mathbf{P}_{\mathbf{L}}$ is tri-diagonal, beginning
\begin{displaymath}\left(\begin{array}{ccccccc} z &
1 & 0 & 0 & 0 & 0 & \ldots \\z & z+1 & 1 & 0 & 0 & 0 & \ldots \\ 0 & 2z & z+2 & 1 & 0 & 0 & \ldots \\ 0 & 0 & 3z & z+3 & 1 & 0 & \ldots
\\0
& 0 & 0 & 4z & z+4 & 1 & \ldots \\0 & 0 & 0 & 0 & 5z & z+5 &\ldots\\ \vdots & \vdots & \vdots & \vdots & \vdots & \vdots &
\ddots\end{array}\right).\end{displaymath}
Now it is well known that $$\sum_{k=0}^n \frac{e_k(z)}{k!} x^k = e^{z(e^x-1)},$$ and hence the polynomials $e_n(z)$ are the moments the
family of orthogonal polynomials whose coefficient array is $\mathbf{L}^{-1}.$
\end{proof}
\begin{corollary} The Hankel transform of $e_n(z)$ is $z^{\binom{n+1}{2}}\prod_{k=1}^n k!$.
\end{corollary}
\begin{proof} From the above, we have that the generating function of $e_n(z)$ is given by the continued fraction
$$\cfrac{1}{1-zx-
\cfrac{zx^2}{1-(z+1)x-
\cfrac{2zx^2}{1-(z+2)x-
\cfrac{3zx^2}{1-\cdots}}}}.$$ In other words, $\beta_n=nz$. Thus
the Hankel transform of $e_n(z)$ is given by
$$\prod_{k=1}^n \beta_k^{n-k+1}=\prod_{k=1}^n (kz)^{n-k+1}=z^{\binom{n+1}{2}}\prod_{k=1}^n k!.$$
\end{proof}
\noindent We note that the Hankel transform of the row sums of $\mathbf{L}=\left[e^{z(e^x-1)},e^x-1\right]$ is equal to
$$(z+1)^{\binom{n+1}{2}}\prod_{k=1}^n k!.$$
\noindent Note also that if we take  $z=e^t$, we obtain a solution to the restricted Toda chain \cite{Toda}.

\section{Proof of Theorem \ref{Thm2} }
\begin{proof}
 We show first that with $\mathbf{L}=\left[\frac{e^{zx}(1-z)}{e^{zx}-ze^x},\frac{e^{x}-e^{zx}}{e^{zx}-ze^x}\right]$, the matrix
 $\mathbf{L}^{-1}$ which is given by
$$\mathbf{L}^{-1}=\left[\frac{e^{zx}(1-z)}{e^{zx}-ze^x},\frac{e^{x}-e^{zx}}{e^{zx}-ze^x}\right]^{-1}=\left[1+zx,
\frac{1}{z-1}\ln\left(\frac{1+zx}{1+x}\right)\right],$$ is the coefficient array of a family of orthogonal polynomials. To this end, we
calculate the production array of $\mathbf{L}=\left[\frac{e^{zx}(1-z)}{e^{zx}-ze^x},\frac{e^{x}-e^{zx}}{e^{zx}-ze^x}\right]$.
We have $f(x)=\frac{e^{x}-e^{zx}}{e^{zx}-ze^x}$, $\bar{f}(x)=\frac{1}{z-1}\ln\left(\frac{1+zx}{1+x}\right)$ and
$$f'(x)=\frac{(1-z)^2e^{x(1+z)}}{(e^{zx}-ze^x)^2}.$$
Thus $$c(x)=f'(\bar{f}(x))=(1+x)(1+zx).$$
Also $g(x)=\frac{e^{zx}(1-z)}{e^{zx}-ze^x}$, which implies that $g'(x)=xf'(x)$ and so
$$r(x)=\frac{g'(\bar{f}(x))}{g(\bar{f}(x))}=z(1+x).$$
Thus the generating function of $\mathbf{P}_{\mathbf{L}}$ is given by
$$e^{tw}(c(w)+t r(w))=e^{tw}((1+w)(1+zw)+t(z(1+w))).$$
Thus the production array $\mathbf{P}_{\mathbf{L}}$ is tri-diagonal, beginning
\begin{displaymath}\left(\begin{array}{ccccccc} z &
1 & 0 & 0 & 0 & 0 & \ldots \\z & 2z+1 & 1 & 0 & 0 & 0 & \ldots \\ 0 & 4z & 3z+2 & 1 & 0 & 0 & \ldots \\ 0 & 0 & 9z & 4z+3 & 1 & 0 &
\ldots \\0
& 0 & 0 & 16z & 5z+4 & 1 & \ldots \\0 & 0 & 0 & 0 & 25z & 6z+5 &\ldots\\ \vdots & \vdots & \vdots & \vdots & \vdots & \vdots &
\ddots\end{array}\right).\end{displaymath} Now it is known that $$\sum_{k=0}^n \frac{\mathsf{EU}_k(z)}{k!} x^k =
\frac{e^{zx}(1-z)}{e^{zx}-ze^x},$$ and hence the polynomials $e_n(z)$ are the moments the family of orthogonal polynomials whose
coefficient array is $\mathbf{L}^{-1}.$
\end{proof}
\begin{corollary} The Hankel transform of $\mathsf{EU}_n(z)$ is $z^{\binom{n+1}{2}}\prod_{k=1}^n k!^2$.
\end{corollary}
\begin{proof} From the above, we have that the generating function of $\mathsf{EU}_n(z)$ is given by the continued fraction
$$\cfrac{1}{1-zx-
\cfrac{zx^2}{1-(2z+1)x-
\cfrac{4zx^2}{1-(3z+2)x-
\cfrac{9zx^2}{1-\cdots}}}}.$$ In other words, $\beta_n=n^2z$. Thus
the Hankel transform of $\mathsf{EU}_n(z)$ is given by
$$\prod_{k=1}^n \beta_k^{n-k+1}=\prod_{k=1}^n (k^2z)^{n-k+1}=z^{\binom{n+1}{2}}\prod_{k=1}^n k!^2.$$
\end{proof}

\bigskip
\hrule
\bigskip
\noindent 2010 {\it Mathematics Subject Classification}: Primary
11B83; Secondary 05A15, 11C20, 15B05, 15B36, 42C05.
\noindent \emph{Keywords:} Integer sequence, exponential Riordan array, Touchard polynomial, exponential polynomial,
moments, orthogonal polynomials, Hankel determinant, Hankel transform.

\bigskip
\hrule
\bigskip
\noindent Concerns sequences \seqnum{A000045},
\seqnum{A000142},
\seqnum{A000262},
\seqnum{A007318},
\seqnum{A008277},
\seqnum{A008292},
\seqnum{A021009},
\seqnum{A094587},
\seqnum{A094816},
\seqnum{A111596}.


\begin{thebibliography}{9}

\bibitem{PasTri} P. Barry, \href{http://www.cs.uwaterloo.ca/journals/JIS/VOL10/Barry/barry202.html}{On a family of generalized Pascal
    triangles defined by
exponential Riordan arrays}, \emph{J. Integer Sequences}, {\bf 10} (2007), Article 07.3.5.

\bibitem{Lah} P. Barry,
\href{http://www.cs.uwaterloo.ca/journals/JIS/VOL10/Barry/barry401.html}{Some Observations on the Lah and Laguerre Transforms of
Integer Sequences}, \emph{J. Integer Sequences}, {\bf 10} (2007), Article 07.4.6.

\bibitem{Toda} P. Barry, The restriced Toda chain, exponential Riordan arrays, and Hankel transforms, preprint, 2010.

\bibitem{BRP}
P. Barry, P. Rajkovic \& M. Petkovic, An application of Sobolev orthogonal polynomials to the computation of a special Hankel
Determinant, in  W. Gautschi, G. Rassias, M. Themistocles (Eds), \emph{Approximation and Computation}, Springer, 2010.

\bibitem{Chihara}
T. S. Chihara,  {\it An Introduction to Orthogonal Polynomials},
Gordon and Breach, New York, 1978.

\bibitem{CRI}
A. Cvetkovi\'c, P. Rajkovi\'c and M. Ivkovi\'c, Catalan
Numbers, the Hankel Transform and Fibonacci Numbers, \emph{Journal of
Integer Sequences}, {\bf 5}, (2002), Article 02.1.3.

\bibitem{DeutschShap} E. Deutsch, L. Shapiro, Exponential Riordan Arrays, Lecture Notes,
Nankai University, 2004, available electronically at \href{http://www.combinatorics.net/ppt2004/Louis%20W.%20Shapiro/shapiro.htm}{\tt
http://www.combinatorics.net/ppt2004/Louis\%20W.\%20Shapiro/shapiro.htm}

\bibitem{ProdMat_0} E. Deutsch, L. Ferrari, and S. Rinaldi,
Production Matrices, \emph{Advances in Applied Mathematics} \textbf{34} (2005) pp.\,101--122.

\bibitem{ProdMat} E. Deutsch, L. Ferrari,
and S. Rinaldi, Production matrices and Riordan arrays, \href{http://arxiv.org/abs/math/0702638v1}{\tt
http://arxiv.org/abs/math/0702638v1}, February 22 2007.

\bibitem{wgautschi} W. Gautschi, {\it Orthogonal Polynomials: Computation and
Approximation}, Clarendon Press, Oxford, 2003.

\bibitem{Kratt} C. Krattenthaler, Advanced Determinant
    Calculus, available electronically at
    \texttt{http://arxiv.org/PS\_cache/math/pdf/9902/9902004.pdf},
    2010.

\bibitem{Kratt1} C. Krattenthaler, Advanced determinant
    calculus: A complement, {\it Linear Algebra and its
    Applications} \textbf{411} (2005) pp1.\,68–-166.

\bibitem{Layman} J. W. Layman, The Hankel Transform and Some
    of
    Its Properties, \emph{Journal of Integer Sequences}, {\bf
    4}, (2001) Article 01.1.5.


\bibitem{P_W} P. Peart, W-J. Woan,
Generating functions via Hankel and Stieltjes matrices, \emph{Journal of Integer Sequences}, \textbf{3} (2000) Article 00.2.1.

\bibitem{Rad} Ch. Radoux, Calcul effectif de certains
    d\'eterminants de Hankel, \emph{Bull. Soc. Math. Belg.},
    XXXI, Fascicule 1, serie B (1979),
    pp.\,49--55.

\bibitem{Roman} S. Roman, \emph{The Umbral Calculus}, Dover
    Publications, 2005.

\bibitem{SGWW} L. W. Shapiro, S. Getu, W-J. Woan and L.C. Woodson,
The Riordan Group, \emph{Discr. Appl. Math.} \textbf{34} (1991)
pp.\,229--239.

\bibitem{Siva} S. Sivasubramanian, Hankel determinants of some sequences of polynomials,
\emph{S\'eminaire Lotharingien de Combinatoire}, \text{63} (2010), Article B63d.

\bibitem{SL1} N. J. A.~Sloane, \emph{The
On-Line Encyclopedia of Integer Sequences}. Published electronically
at \texttt{http://www.research.att.com/$\sim$njas/sequences/}, 2010.

\bibitem{SL2} N. J. A.~Sloane, The On-Line Encyclopedia of Integer
Sequences, \emph{Notices of the AMS}, \textbf{50} (2003), pp.\, 912--915.

\bibitem{Spru} R. Sprugnoli, Riordan arrays and combinatorial sums,
\emph{Discrete Math.} \textbf{132} (1994) pp.\,267--290.

\bibitem{Szego} G. Szeg\"o, \emph{Orthogonal Polynomials}, 4th ed.
Providence, RI, Amer. Math. Soc., (1975)

\bibitem{Bell} E. W. Weisstein, Bell Polynomial, \emph{From MathWorld--A Wolfram Web Resource}, available electronically at  { \tt
    http://mathworld.wolfram.com/BellPolynomial.html}, 2010.


\bibitem{Wall} H.~S. Wall, \emph{Analytic Theory of
    Continued Fractions}, AMS Chelsea Publishing, (2000)

\end{thebibliography}
\end{document}